\documentclass[11pt,a4]{amsart}

\usepackage{amsmath}
\usepackage{amssymb}
\usepackage{cancel}
\usepackage{graphicx}
\usepackage{url,hyperref}
\usepackage[style=american]{csquotes}
\usepackage{enumerate}
\usepackage[width=\textwidth]{caption}
\usepackage{subcaption}
\DeclareCaptionSubType*[arabic]{figure}
\usepackage[normalem]{ulem}

\usepackage{own-macros}

\begin{document}

\title[Is a complete, reduced set necessarily of constant width?]{Is a complete, reduced set necessarily of constant width?}

\author{Ren\'e Brandenberg, Bernardo Gonz\'alez Merino, Thomas Jahn, and Horst Martini}
\address{Zentrum Mathematik, Technische Universit\"at M\"unchen, Boltzmannstr. 3, 85747 Garching bei M\"unchen, Germany} \email{brandenb@ma.tum.de}
\email{bg.merino@tum.de}
\address{Fakult\"at f\"ur Mathematik, Technische Universit\"at Chemnitz, 09107 Chemnitz, Germany} \email{thomas.jahn@mathematik.tu-chemnitz.de} \email{martini@mathematik.tu-chemnitz.de}


\thanks{The second author is partially supported by
Consejer\'ia de Industria,
Turismo, Empresa e Innovaci\'on de la CARM through Fundaci\'on
S\'eneca, Agencia de Ciencia y Tecnolog\'ia de la Regi\'on de
Murcia, Programa de Formaci\'on Postdoctoral de Personal
Investigador
and project 19901/GERM/15, Programme in Support of Excellence
Groups of the Regi\'{o}n de Murcia, Spain, and by MINECO project reference MTM2015-63699-P, Spain.}

\subjclass[2010]{46B20, 52A20, 52A21, 52A40, 52B11}

\keywords{bodies of constant width, Bohnenblust's inequality, complete bodies, gauges, generalized Minkowski spaces, Leichtweiss' inequality, normed spaces, perfect norms, reduced bodies}

\begin{abstract}
  Is it true that a convex body $K$ being complete and reduced with respect to some
  gauge body $C$ is necessarily of constant width, \ie, satisfies $K-K=\rho(C-C)$ for some $\rho>0$?
  We prove this implication for several cases including the following:  if $K$ is a simplex
  and or if $K$ possesses a smooth extreme point, then the implication holds. Moreover, we derive several new results on perfect
  norms.
\end{abstract}

\maketitle

\section{Introduction} \label{s:intro}

The notions of constant width and completeness are well known in fields like convexity, Banach space theory, and convex analysis.
A compact, convex set $K$ in $\R^n$ (i.e., a convex body) is said to be
of \emph{constant width} if the distance of any two parallel supporting hyperplanes of $K$ is the same.
On the other hand, such a convex body $K$ is called (diametrically) \emph{complete} if any proper superset of it has larger
diameter than $K$. It is obvious that these definitions can also be used in any (normed or) Minkowski
space, using the corresponding distance measures.
In Euclidean spaces of any dimension as well as in arbitrary normed planes
  constant width and completeness are equivalent. This is no longer the case in $n$-dimensional Minkowski spaces if $n>2$,
yielding the notion of perfect norms (used for norms in which this equivalence still holds).
Surveys and basic references on bodies of constant width and complete bodies in Euclidean $n$-space are \cite{Cha-Gro}, \cite{Gro},
and \cite{Hei-Mar}, and results on their analogues in Minkowski spaces are collected or proved in \cite{Eg}, \cite{Mar-Swa},
\cite{MSch}, and \cite{MSch2}.
A relatively new and closely related notion is that of reduced bodies: A convex body $K$ in $\R^n$ is said to be \emph{reduced}
if any convex proper subset of it has smaller (minimal) width.
This notion creates already in Euclidean $n$-space sufficiently interesting open research problems
(see the survey \cite{Lass-Mart}). For example, it is unknown whether there exist reduced polytopes
in Euclidean $n$-space if $n>2$ (\cf~\cite{AvMa} and \cite{Lass-Mart}). As first shown in \cite{LM}, the notion of reducedness can be carried over to Minkowski spaces, too; existing results and interesting research problems are presented in \cite{LM2}.
Since, in general, in such spaces complete bodies need not be reduced (see \cite{Mar-Wu})
and one can construct reduced bodies which are not complete, the question is how these two classes are related to each other.
Indeed, the family of bodies of constant width forms a subfamily of both. In this article, for the first time the
question is posed whether for non-perfect norms the family of bodies of constant width forms the intersection of the two other
families! Moreover, this question is investigated even for
\emph{generalized Minkowski spaces}, in which the
 unit balls (called gauge bodies) are still convex bodies having the origin as interior point,
 but need not be centrally symmmetric.
 So our paper contains also various new notions which are interesting for themselves and necessary for switching from normed spaces to generalized Minkowski spaces.

To do so, we use the so called Minkowski asymmetry several times.
There exists a rich variety of asymmetry measures for convex sets
(see \cite[Sect. 6]{Gr} for the possibly most comprehensive
overview), but amongst all, the one receiving most attention is the
Minkowski asymmetry.
In \cite{BrGo2} it is shown how it naturally relates
to complete and constant width sets in Minkowski spaces.
Moreover, it has been repeatedly used to sharpen and strengthen geometric inequalities and related results, \cf~\cite{BeFr,BrK2,GuKa,GoLiMePa}.

We prove the following results for generalized Minkowski spaces: If the convex body $K$ is a complete and reduced simplex, then it is of constant width. And the same implication holds for the large family of all convex bodies possessing a smooth extreme point (obviously, this class contains all strictly convex and all smooth convex bodies). Extending the notion of perfect norm to generalized Minkowski spaces, we also obtain some results on perfect gauge bodies, including a characterization of them via completions of the convex bodies under consideration.

\section{Notation and background}

By $\conv(A)$, ${\rm int}(A)$, and $\aff(A)$ we denote the \cemph{dred}{convex hull}, the \cemph{dred}{interior}, and the \cemph{dred}{affine hull} of a set $A\subset\R^n$,
respectively, and we will write $[x,y]=\conv(\{x,y\})$ for the line segment whose endpoints are $x,y\in\R^n$. We also use the notation $[n]$ for $\{1,\dots,n\}$.

Let $\CK^n$ be the family of convex and compact sets (bodies) in $\R^n$, and let $C,K \in\CK^n$.
We call $K+C :=\{x+y:x\in K,\,y\in C\}$ the
\cemph{dred}{Minkowski sum} of $K$ and $C$, and for any $\rho > 0$, $\rho K := \{\rho x : x\in K\}$ is the \cemph{dred}{$\rho$-dilatation} of $K$;
we write $-K:=(-1)K$.

If $C \in \CK^n$ with the \emph{origin} $0 \in \inte(C)$, then
$C$ may be called a \cemph{dred}{gauge body} (or \emph{unit ball}) of a \emph{generalized Minkowski space} induced by $C$.
Any non-negative function $\gamma$,  which takes the value $0$ only at the origin and
satisfies $\gamma(\lambda x) = \lambda \gamma(x)$ for all $\lambda \ge 0$
and $\gamma(x+y) \le \gamma(x)+\gamma(y)$, is called a \cemph{dred}{gauge function}.
Thus a gauge function meets all the requirements of a norm except for the symmetry $\gamma(x) = \gamma(-x)$.

This means that the definitions $\gamma(x) := \inf \{\lambda \ge 0 | x \in \lambda C\}$
for any given gauge body $C$
and $C := \{x \in \R^n | \gamma(x) \le 1\}$ for any given gauge function $\gamma$
establish a one-to-one correspondence between gauge bodies and gauge functions similar to the well known
  one-to-one correspondence between norms and $0$-symmetric bodies with non-empty interior.
Note that in the following we do not really assume $0$ to be an interior point of $C$, as we
do not use the gauge function; our considerations are, more generally, based on translation-invariant radius functions.

Denoting the Hausdorff distance by $d_H$, we say that a sequence $(A_i)_{i \in \N}$ with $A_i \subset \R^n$
  converges to $A \subset \R^n$ if $\lim_{i\rightarrow\infty} d_H(A_i,A)=0$.

The \cemph{dred}{support function} $h(K,\cdot) : \R^n \to \R$ of a convex body $K$ is defined by
$h(K,a)=\max\{a^T x : x\in K\}$, $a\in\R^n$. For $b\in\R$, we write $H^\le_{a,b}:= \{x\in\R^n:a^Tx \le b\}$
for the half-space with outer normal $a$ and offset $b$, and $H_{a,b}:= \{x\in\R^n:a^Tx=b\}$ is written
for the corresponding boundary hyperplane.
The hyperplane $H_{a,b}$ \textit{supports} $K$ at $x\in K$ if $x\in H_{a,b}$ and $K\subset H_{a,b}^\le$,
which means that $b=h(K,a)$.
A point $x\in \bd(K)$ is \cemph{dred}{extreme} if $x \not\in \conv(K\setminus\{x\})$,
and \cemph{dred}{smooth} if there exists a unique supporting hyperplane supporting $K$ at $x$.
The set of all extreme points of $K$ is denoted by $\ext(K)$.

The term $K \subset_t C$ abbreviates that there exists a translation $c \in \R^n$ such that $K \subset c+C$, and the
term $K \subset^{\opt} C$ summarizes that $K \subset C$ and for all $\rho<1$ it holds that $K \not\subset_t \rho C$.
The \cemph{dred}{circumradius} $R(K,C)$ of $K$ with respect to $C$ is the smallest $\lambda\ge 0$
such that a translate of $\lambda C$ contains $K$, \ie, there exists $c \in \R^n$ such that $K \subset^{\opt} c + R(K,C) C$.
Now, the \cemph{dred}{inradius} $r(K,C)$ of $K$ with respect to $C$ is the largest $\lambda\ge 0$ such that
a translate of $\lambda C$ is contained in $K$.
The translations needed above are called the \cemph{dred}{circumcenter} and the \cemph{dred}{incenter} of $K$ with respect to $C$,
respectively.

The \cemph{dred}{diameter}
$D(K,C)$ of $K$ with respect to $C$ is defined as $D(K,C)=2\max\{R([x,y],C):x,y\in K\}$,
the \cemph{dred}{$s$-breadth} $b_s(K,C)$ (often also called \cemph{dred}{$s$-width}) of $K$ with respect to $C$
in direction of $s \in \R^n\setminus \{0\}$ is $b_s(K,C)=2h(K-K,s)/h(C-C,s)$.
We will use several times that $D(K,C) = \max_{s\neq 0}b_s(K,C)$, which is shown to be true in \cite{GK}
for $C$ symmetric, but obviously remains true, since both, diameter and $s$-breadth, keep
constant when replacing $C$ by $1/2(C-C)$ (see \cite[Lemma 2.8]{BrK2}).
The \cemph{dred}{(minimal) width} $w(K,C)$ of $K$ with respect to $C$
is $w(K,C)=\min_{s\neq 0}b_s(K,C)$.

The \cemph{dred}{Minkowski asymmetry} $s(K)$ is the smallest
$\lambda \ge 0$ such that $\lambda K$ contains a translate of $-K$, \ie, $s(K) =R(-K,K)$. Moreover, if for $c\in\R^n$
the inclusion $-(-c+K)\subset s(K)(-c+K)$ holds, we say that $c$ is the \cemph{dred}{Minkowski center} of $K$, and if $c=0$,
we say that $K$ is \cemph{dred}{Minkowski-centered}.


A set $K$ is \cemph{dred}{complete} with respect to $C$ if $D(K',C)>D(K,C)$ for every $K' \supsetneq K$,
and $K$ is \cemph{dred}{reduced} with respect to $C$ if $w(K',C)<w(K,C)$,
for every $K' \subsetneq K$.
With $K^*\supset K$ and $K_* \subset K$ we denote a \cemph{dred}{completion} or a \cemph{dred}{reduction} of $K$, respectively.
This means in the first case that $D(K,C)=D(K^*,C)$ as well as $K^*$ is complete, and in the second that $w(K_*,C)=w(K,C)$ as well as $K_*$ is reduced.
A set $K$ is of \cemph{dred}{constant width} with respect to $C$ if $w(K,C)=D(K,C)$ or, equivalently, if $K-K=\rho(C-C)$
(where $\rho=1/2 D(K,C)$ in this case).
A set $K$ is called \cemph{dred}{pseudo-complete} with respect to a centrally symmetric $C$ if
$D(K,C)=r(K,C)+R(K,C)$. Recognize that if $C$ is centrally symmetric, then
any complete $K$ is also pseudo-complete (see \cite{MSch}).

A gauge body $C\in\CK^n$ and the generalized Minkowski space induced by $C$ are called \cemph{dred}{perfect} if $K$ is of constant width with respect to $C$ whenever $K$ is complete with respect to $C$.
By definition, in case that $C=-C$, the norm induced by $C$ is called \emph{perfect} iff $C$ is perfect.

By an \cemph{dred}{$n$-simplex} we denote the convex hull of $n+1$ affinely independent points.

\section{
Completeness and reducedness}\label{sec:probres}

In Euclidean spaces of arbitrary dimension and in normed planes completeness and constant width are equivalent notions (see \cite{BF}, \cite{Eg}, and \cite{Mar-Swa}, as well as
Lemma \ref{l:polyimplgen} below). Moreover, it is easy to see that any $K$ of constant width with respect to an arbitrary
body $C$ is complete and reduced with respect to $C$. However, the contrary is, to the best of our knowledge,
not known in general and has not been asked before, and it is the backbone of this article.

\begin{open}\label{op:compltred}
Let $K,C\in\CK^n$ be such that $K$ is complete and reduced with respect to $C$. Does this imply that
$K$ is of constant width with respect to $C$?
\end{open}

The following lemma collects some facts about completeness and reducedness, showing that most of the problems
may be reduced from arbitrary bodies to symmetric ones.

\begin{lem}\label{l:gauge}
Let $K,C\in\CK^n$. Then the following statements hold true.
\begin{enumerate}[(i)]
\item\label{p:KcwC-C} $K$ is of constant width with respect to $C$ iff $K$ is of constant width with respect to $C-C$.
\item\label{p:KcomC-C} $K$ is complete with respect to $C$ iff $K$ is complete with respect to $C-C$.
\item\label{p:KredC-C} $K$ is reduced with respect to $C$ iff $K$ is reduced with respect to $C-C$.
\item\label{p:PerfectC-C} $C$ is perfect iff $C-C$ is perfect.
\item\label{p:existCR} There exist completions and reductions of $K$ with respect to $C$.
\item\label{p:CompExt} If $K$ is complete with respect to $C$, then every point $x\in\bd(K)$
is the endpoint of a diametrical segment.
\item\label{p:ReduExt} If $K$ is reduced with respect to $C$, then for every $x\in \ext(K)$
there exist $y_x\in K$ and $s\in\R^n\setminus\{0\}$ such that $b_s([x,y_x],C)=b_s(K,C)=w(K,C)$
(see \cite[Theorem 1]{LM2} for the case that $C=-C$).
\item\label{p:SIP} The set $K$ is complete with respect to $C$ iff $K = \bigcap_{x \in \bd(K)} (x + D(K,C-C)(C-C))$ (spherical
intersection property with respect to $C-C$).
\end{enumerate}
\end{lem}

\begin{proof}
The first statement directly follows from the fact that $K$ is of constant width with respect to $C$ iff $K-K = 1/2D(K,C)(C-C)$,
and \eqref{p:KcomC-C} as well as \eqref{p:KredC-C} directly follow from the fact that
$w(K,C)=2w(K,C-C)$ and $D(K,C)=2D(K,C-C)$.
The fourth statement is a direct corollary out of \eqref{p:KcwC-C} and \eqref{p:KcomC-C}, while
the others follow from \eqref{p:KcomC-C} and \eqref{p:KredC-C},
taking into account that
all those statements are well known for the case that $C = -C$ (see \cite{Eg}, \cite{LM2}, and \cite{Mar-Swa}).
\end{proof}

The following proposition characterizes an optimal containment between two sets
by their touching points (\cf~Theorem 2.3 in \cite{BrK}).

\begin{prop}\label{prop:opt_hom}
  Let $K,C \in {\mathcal K}^n$. We have $K \subset^{\opt}C$ iff $K\subset C$ and for some $2\leq m\leq n+1$,
  there exist $p^1,\dots,p^m\in K \cap C$ and hyperplanes $H_{a^i,1}$ supporting $K$ and $C$ in $p^i$, $i \in [m]$,
  such that $0\in\conv(\{a^1,\dots,a^m\})$.
\end{prop}

  The following corollary combines optimal containment with the notion of Minkowski-centered polytopes.

\begin{cor}\label{cor:AsymCont}
Let $P\in\CK^n$ be a Minkowski-centered polytope.
Then
\[
\left(1+\frac{1}{s(P)}\right)\conv(P\cup(-P))\subset^{\opt} P-P \subset^{\opt} (s(P)+1)(P\cap(-P))\,,
\]
and there exist vertices $p^i$ and facet normals $a^i$ of $P$, with $i \in [m]$ for some $2\leq m\leq n+1$,
such that $0\in\conv(\{a^1,\dots,a^m\})$ and $\pm (1+1/s(P))p^i$ is a vertex of
$\left(1+1/s(P)\right)\conv(P\cup(-P))$ contained in a facet of $P-P$,
which itself is completely contained in a facet of $(s(P)+1)(P\cap(-P))$, both with outer normal $\mp a^i$.
\end{cor}

\begin{proof}
Since $0$ is the Minkowski center of $P$, we have $-P \subset^{\opt} s(P)P$.
Thus, by Proposition \ref{prop:opt_hom}, there exist vertices $p^i$ of $P$ and
$a^i\neq 0$, $i \in [m]$, satisfying $0\in\conv(\{a^1,\dots,a^m\})$, such that $F_i= H_{a^i,1}\cap P$ is a facet of $P$ with
$-p^i\in s(P) F_i$ for all $i \in [m]$, for some $m \in \{2,\dots,n+1\}$.
Now, it obviously holds that $\pm F_i':=\pm(F_i-p^i)$ is a facet of $P-P$ containing
the vertex $\pm(1+1/s(P))p^i$ of $(1+1/s(P))\conv(P\cup(-P))$, which is also contained in the facet
$\pm(s(P)+1) (F_i\cap(-P))$ of $(1+s(P))(P\cap(-P))$. This proves that the latter is a superset of $F_i'$ using the
containment of $P-P$ in $(s(P)+1)(P\cap(-P))$.
\end{proof}

The containment chain
\[
\left(1+\frac{1}{s(K)}\right)\conv(K\cup(-K)) \subset^{\opt} K-K \subset^{\opt} (s(K)+1)(K\cap(-K))
\]
in Corollary \ref{cor:AsymCont} remains true for non-polytopal $K\in\CK^n$.
However, the \enquote{facet-facet touching} of  $K-K$ and $(s(K)+1)(K\cap(-K))$ described in the corollary gets lost.

\medskip

If $K \subset \R^3$ is a regular tetrahedron with centroid at the origin, Corollary \ref{cor:AsymCont} explains
how the cube $\conv(K\cup(-K))$, the cuboctahedron $K-K$, and the octahedron $K\cap(-K)$
can be placed such that the cube is optimally contained in the octahedron, and still the cuboctahedron fits in between
(\cf~Figure \ref{1}).

\begin{figure}[h]
      \includegraphics[width=6cm]{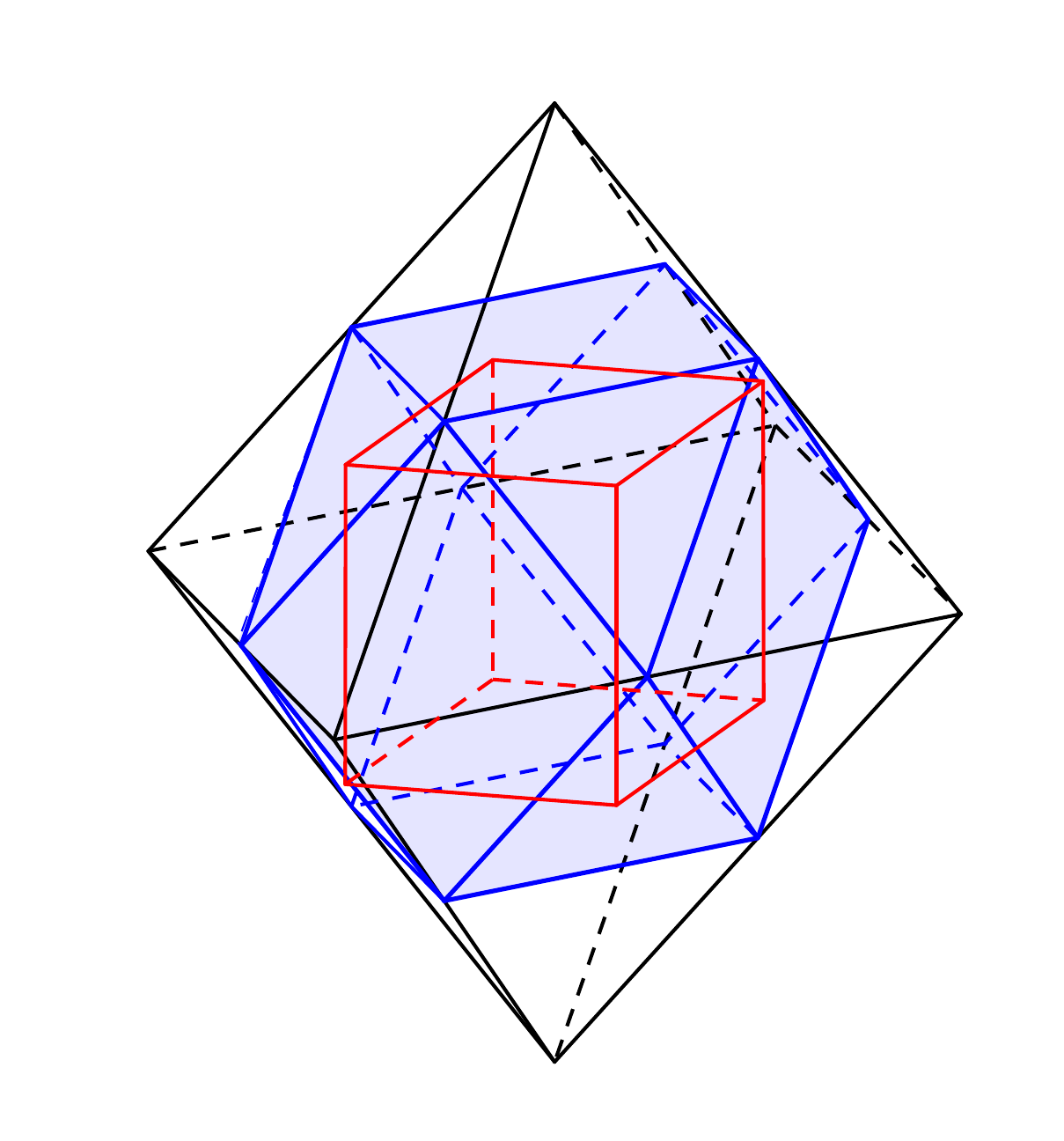}
            \caption{A cube optimally contained in an octahedron, and a cuboctahedron fitting in between.}\label{1}
  \end{figure}

\medskip

Next we state two propositions taken from \cite[Lemma 2.5 and Corollary 2.10]{BrGo2} and characterizing pseudo-completeness
(which becomes completeness in the simplex case).


\begin{prop}\label{th:pseudocomplete}
  Let $K,C\in\CK^n$ with $K$ being Minkowski-centered and $C=-C$. Then the following are equivalent:
\begin{enumerate}[(i)]
\item $K-K \subset D(K,C)C \subset (s(K)+1)(K \cap (-K))$.
\item $K$ is pseudo-complete with respect to $C$, i.e., $D(K,C)=r(K,C)+R(K,C)$.
\end{enumerate}
Moreover, if $K$ is complete with respect to $C$, then $K$ satisfies any of the conditions above,
and any of them implies $R(K,C)/D(K,C)=s(K)/(s(K)+1)$ (\cf \cite[Corollary 6.3]{BrK2}).
\end{prop}


\begin{prop}\label{prop:Tcomplete}
Let $S,C\in\CK^n$ with $S$ being a Minkowski-centered $n$-simplex and $C=-C$.
Then the following are equivalent:
\begin{enumerate}[(i)]
\item $D(S,C)=r(S,C)+R(S,C)$.
\item $S-S\subset D(S,C)C \subset (n+1)(S \cap (-S))$.
\item $S$ is complete with respect to $C$.
\item $R(S,C)/D(S,C)=n/(n+1)$ (equality case in Bohnenblust's inequality, \cf~\cite{Bo,Le}).
\end{enumerate}
\end{prop}

The proposition below is taken from \cite[Corollary 7]{LM2} and shows a quite similar structure for the reducedness
of simplices as the one given in Proposition \ref{prop:Tcomplete} for completeness.

\begin{prop}\label{prop:redS}
Let $S,C\in\CK^n$ with $S$ being an $n$-simplex and $C=-C$. Then the following are equivalent:
\begin{enumerate}[(i)]
\item $S$ is reduced with respect to $C$.
\item $w(S,C)C\subset S-S$ touches all facets of $S-S$ with outer normals parallel to outer normals of facets of $\pm S$.
\end{enumerate}
\end{prop}

Putting these two propositions referring to simplices together, we obtain our first theorem.

\begin{thm}
  Let $S,C\in\CK^n$, such that $S$ is a complete and reduced simplex with respect to $C$. Then $S$ is of constant width with respect to $C$.
\end{thm}

\begin{proof}
  Without loss of generality, we can assume that $C$ is $0$-symmetric
  (see Lemma \ref{l:gauge} \eqref{p:KcwC-C}, \eqref{p:KcomC-C}, and \eqref{p:KredC-C}).
  The completeness of $S$ implies by Proposition \ref{prop:Tcomplete} that
  \[
  S-S\subseteq D(S,C)C\subseteq(n+1)(S\cap(-S))\,,
  \]
  and by Corollary \ref{cor:AsymCont} all facets of $S-S$ parallel to facets of $S$
  are contained in facets of $(n+1)(S\cap(-S))$.
  On the other hand, since $S$ is reduced, Proposition \ref{prop:redS} implies that
  $w(S,C)C\subseteq S-S$ with touching points in all facets of $S-S$ which are
  parallel to facets of $S$. Hence $w(S,C)C\subset^{\opt}D(S,C)C$, and thus $w(S,C)=D(S,C)$.
\end{proof}

There is a natural connection between the equality case in the inequality of Leichtweiss (see \cite{Le}) and reduced sets,
which is reflected in the following proposition.

\begin{prop}\label{prop:leichtT}
Let $S$ be a Minkowski-centered $n$-simplex. Then the following are equivalent:
\begin{enumerate}[(i)]
\item $w(S,C)/r(S,C)=n+1$ (equality case in Leichtweiss' inequality \cite{Le}).
\item $(1+1/n)\conv(S \cup (-S)) \subset w(S,C)C \subset S-S$, and $(1+1/n)S$ touches $S-S$ in all facets
with outer normals parallel to outer normals of facets of $\pm S$ in the points precisely given
by Corollary \ref{cor:AsymCont}.
\end{enumerate}
\end{prop}

\begin{rem}
It is immediate to observe that any simplex $S$, satisfying any condition in Proposition \ref{prop:leichtT},
fulfils also Proposition \ref{prop:redS}, which means that it is reduced. However, the contrary is not true, \ie,
there exist reduced simplices $S$ such that $w(S,C)/r(S,C) < n+1$.
\end{rem}



\begin{open}\label{op:realquestion}
If $K$ is complete and reduced with respect to $C$ for a Minkowski-centered convex body $K$, does this imply
\[\left(1+\frac{1}{s(K)}\right)\conv(K\cup(-K))\subset w(K,C)C?\]
\end{open}

If the answer to Open Question \ref{op:realquestion}  would be yes, then,
together with Proposition \ref{th:pseudocomplete} and the trivial containment $w(K,C)C\subset K-K$,
by Corollary \ref{cor:AsymCont} we would have that $w(K,C)C\subset^{\rm opt}D(K,C)C$,
and thus the answer to Open Question \ref{op:compltred} would be yes, too.

\begin{rem}
We offer another question, motivated by Proposition \ref{prop:redS}, which is maybe somehow
\enquote{closer to reality}: if $K$ is reduced
with respect to $C$, does there exist some set $K'$, related to $K$, such that $K'\subset w(K,C)C\subset K-K$
with nice \enquote{touching conditions}? For instance, remember that if $K$ is
a reduced simplex with respect to to a 0-symmetric $C$, then Proposition \ref{prop:redS} implies that
$w(K,C)C$ touches all facets of $K-K$ in points $\pm p^i$, $i\in[n+1]$, which are parallel to facets of $K$,
\ie, we could define $K':=\conv(\{\pm p^i:i\in[n+1]\})$), and $K'$ and $(1+1/n) \conv(K\cup(-K))$ would not be necessarily equal.
\end{rem}

The two subsequent propositions are taken from \cite[Corollary 1]{LM2} and \cite[Lemma 4]{MSch2}, respectively.

\begin{prop}\label{prop:Creduc}
Let $C=\conv(\{\pm q^1,\dots,\pm q^m\})$ be a 0-symmetric polytope,
and $K\in\CK^n$ be reduced with respect to $C$. Then $K$ is a polytope there exist $c^i \in \R^n$ such that
\[
K=\conv(\{c^i+(w(K,C)/2)[-q^i,q^i] : i \in [m]\})\,,
\]
and each segment $c^i+(w(K,C)/2)[-q^i,q^i]$ attains the width $w(K,C)$
of $K$ with respect to $C$.
\end{prop}

\begin{prop}\label{prop:MSch}
  Let $C=\bigcap_{j\in[l]}H_{\pm a^j,1}^{\leq}$, $a^j\in\R^n$,
  $j\in[l]$, be a 0-symmetric polytope, and the polytope $K$ be complete
  with respect to $C$. Then there exist $d^j\in\R^n$, $j\in[l]$, such that
  \[
    K=\bigcap_{j\in[l]}\left(d^j+H_{\pm a^j,1}^{\leq}\right),
  \]
  and the diameter $D(K,C)$ is attained in every direction $a^j$, $j\in[m]$.
\end{prop}

\begin{rem}\label{rem:ComplRedPolyt}
  It follows directly from the definition of the width that $w(K,C)C\subset^{\rm opt} K-K$ for
  all 0-symmetric $C$. Now, by Proposition \ref{prop:Creduc},
  \emph{all} vertices  $\pm w(K,C)q^i$ of $w(K,C)C$ belong to $\bd(K-K)$
  if $C$ is a 0-symmetric polytope and $K$ is reduced with respect to $C$ (if not, then one of the
  segments $c^i+(w(K,C)/2)[-q^i,q^i]$ could not attain the width of $K$
  in any direction). In this sense, Proposition \ref{prop:Creduc} strengthens
  the general containment $w(K,C)C\subset^{\rm opt} K-K$, which only assures a certain distribution of the touching vertices.

  Analogously, we have $K-K\subset^{\rm opt}D(K,C)C$ if $C$ is $0$-symmetric in general, and
  if $C$ is a 0-symmetric polytope and $K$ is complete with respect to $C$, then it follows from Proposition \ref{prop:MSch}
  that $K$ is a polytope and \emph{all} facets of $K-K$ which are parallel to facets of $K$
  are contained in facets of $D(K,C)C$. Thus, we obtain again a strengthening of the optimal containment
  $K-K\subset^{\opt} D(K,C)C$ in the general case (cf. Propositions \ref{prop:opt_hom} and \ref{th:pseudocomplete}), which only assures, a certain distribution of the touching points between $K-K$ and $D(K,C)C$.

  Generalizing this to gauge bodies $C$ which are possibly not $0$-symmetric, we obtain from combining Lemma \ref{l:gauge} and the two
  Propositions \ref{prop:Creduc} and \ref{prop:MSch}
  that we may just replace the vertices/facets of $C$ by those of $C-C$ in the representation of a reduced/complete $K$, respectively.
  However, while the vertices of $C-C$ are all obtained from simply taking differences of vertices of $C$, the facets of $C-C$ may come from
  any pairs of subdimensional faces of $C$ lying in antipodal supporting hyperplanes and having at least $n-1$ as sum of their dimensions.
\end{rem}

\begin{lem}
If $C$ is a 0-symmetric polytope and $K$ is complete and reduced with respect to $C$,
such that there exists a vertex of $w(K,C)C$ belonging to a facet of $K-K$
parallel to a facet of $K$, then $K$ is of constant width with respect to $C$.
\end{lem}

\begin{proof}
This follows directly from 
Remark \ref{rem:ComplRedPolyt} as the completeness implies that
facets of $K-K$ parallel to those of $K$ have to be contained in facets of $D(K,C)C$.
Hence $w(K,C)C\subset^{\opt}D(K,C)C$, and thus $w(K,C)=D(K,C)$.
\end{proof}

\begin{lem}
Let $C$ be a 0-symmetric polytope and $K$ be complete and reduced with respect to $C$, with all the notation used in
Proposition \ref{prop:Creduc} and Proposition \ref{prop:MSch}.
If there exist $i\in[m]$ such that $c^i+(w(K,C)/2)q^i$
is a vertex of $K$ and $c^i-(w(K,C)/2)q^i$ belongs to the relative interior
of a facet $F^j=(d^j+H_{a^j,1})\cap K$ of $K$, $j\in[l]$, then $K$ is of constant width.
\end{lem}

\begin{proof}
Proposition \ref{prop:Creduc} implies that the segment $c^i+(w(K,C)/2)[-q^i,q^i]$
attains the width $w(K,C)$ in some direction. Since $c^i-(w(K,C)/2)q^i$ belongs to the relative interior
of $F_j$, then $w(K,C)=w(c^i+(w(K,C)/2)[-q^i,q^i],C)=b_{a^j}(K,C)$.
On the other hand, since $K$ is complete, every point in $\bd(K)$
is an endpoint of a diametrical segment (\cf~ Lemma \ref{l:gauge} (vi)), thus also $c^i-(w(K,C)/2)q^i$.
However, since the only hyperplane supporting $K$ at $c^i-(w(K,C)/2)q^i$ is $d^j+H_{a^j,1}$, we obtain $D(K,C)=b_{a^j}(K,C)$,
and therefore $w(K,C)=D(K,C)$.
\end{proof}

Now we are able to state our second theorem confirming Open Question \ref{op:compltred}
in a great variety of situations, for instance, for any smooth or for any strictly convex body.

\begin{thm}\label{th:extrsmooth}
If $K$ is complete and reduced with respect to $C$ and there exists $x\in \ext(K)$ as a smooth boundary point of $K$, then $K$ is of
constant width.
\end{thm}

\begin{proof}
By Lemma \ref{l:gauge} \eqref{p:KcwC-C}, \eqref{p:KcomC-C} and \eqref{p:KredC-C},
we can assume that $C$ is centrally symmetric.
Since $K$ is complete and $x\in \bd(K)$, we may use Lemma \ref{l:gauge} \eqref{p:CompExt}
to obtain that there exists $y_x\in K$ such that $2R([x,y_x],C)=D(K,C)$.

On the other hand, Lemma \ref{l:gauge} \eqref{p:ReduExt} implies that
there exist two parallel supporting hyperplanes $H_{\pm a,\beta_i}$, $i=1,2$, $a \neq 0$, $\beta_i \in \R$, at distance $w(K,C)$
such that $x\in H_{a,\beta_1}$.

Now, since $[x,y_x]$ is a diametrical segment, there exists $s\neq 0$ such that $b_s(K,C)=D(K,C)$.
Applying the smoothness of $K$ at $x$, we obtain $s=\lambda a$, $\lambda > 0$, and therefore $w(K,C)=D(K,C)$.
\end{proof}

An example not covered by Theorem \ref{th:extrsmooth} is the following: Let $K \in \R^3$ be the convex hull of a two-dimensional disc and a segment orthogonal and not disjoint to it. Then all extreme points of
 $K$ are non-smooth, even though $K$ is not a polytope.

\section{Perfect gauge bodies}

In the following, we connect the Open Question \ref{op:compltred} with perfect gauge bodies (and thus
with perfect norms in case that these bodies are $0$-symmetric). The following observation is clear.

\begin{rem}
If the Open Question \ref{op:compltred} holds true, then
a gauge body is perfect if and only if completeness implies reducedness.
\end{rem}

The final lemma in \cite{Eg} gives in its negation a necessary condition for a $3$-dimensional polytopal
norm to be perfect:

\begin{prop}\label{prop:Eggleston}
  Any perfect $0$-symmetric polytopal body $C\in\CK^3$ is simple (\ie, every vertex of $C$ is contained in at most three facets).
\end{prop}

The following lemma extends Eggleston's result (see \cite{Eg}) to higher-dimensional spaces.

\begin{lem}\label{l:ExtEggleston}
  Let $n \ge 3$ and $C\in\CK^n$ be a $0$-symmetric polytope. If $C$ is perfect, then every pair of non-disjoint
  facets $F_1, F_2$ of $C$ intersects in at least an edge of $C$.
\end{lem}

\begin{proof}
Let us assume that $F_1, F_2$ are two facets of $C$ only intersecting in a vertex $v$ of $C$.
The intersection $\aff(F_1)\cap\aff(F_2)$ is an affine $(n-2)$-subspace.
Let $a\in\R^n$, and $H_{a,1}$ be a hyperplane supporting $C$ solely in $v$ and containing
$\aff(F_1)\cap\aff(F_2)$, such that $C\subset H_{a,1}^\le$. Let $\eps>0$ be small enough such that $H_{a,1-\eps}$
intersects $F_1$ and $F_2$ in $(n-2)$-dimensional polytopes and $x^i \in \relint(F_i \cap H_{a,1-\eps})$, $i=1,2$.
Then $H_{a,1-\eps}$ is a supporting hyperplane of the intersection $C\cap(x^2-x^1+C)$ and $X:= C \cap(x^2-x^1+C) \cap H_{a,1\eps}$
an $(n-2)$-dimensional polytope.

We now define the set $Y:=\conv(\{0,x^2-x^1\}\cup X)$.
If $\eps$ tends to 0, this set $[0,x^2-x^2]$ and $X$ converge to $0$ and $v$, respectively.
On the other hand, for any $x\in X\subset F_1\subset\bd(C)$,
we have that $[-x,x]\subset^{\rm opt}C$, hence $D([0,x],C)=1$. Analogously,
$x^1-x^2+x\in x^1-x^2+X\subset F_2\subset \bd(C)$, thus $[x^1-x^2+x,-(x^1-x^2+x)]\subset^{\rm opt}C$,
and hence $D([x^2-x^1,x],C)=D([x^1-x^2+x,0],C)=1$.
Since the diameter $D(Y,C)$ is always attained by a pair of extreme
points, we conclude that $D(Y,C)=1$.

Now let $Y^*$ be a completion of $Y$ with respect to $C$. Using the spherical intersection property (Lemma \ref{l:gauge} \eqref{p:SIP}),
we obtain
\[
Y^*=\bigcap_{x\in Y^*}(x+C)\subset C\cap(x^2-x^1+C),
\]
and since $H_{a,1-\eps}$ supports $C\cap(x^2-x^1+C)$ in $X$, it also supports $Y^*$ in $X$.
Hence the supporting hyperplanes of $Y^*-Y^*$ parallel to $H_{a,1-\eps}$
support it in a set of dimension at least $n-2$, while $C$ is only supported at $\pm v$. This proves $Y^*-Y^* \neq C$, and hence
$Y^*$ is complete but not of constant width with respect to $C$.
\end{proof}

\begin{rem}\label{rem:polytope}
  Lemma \ref{l:ExtEggleston} extends Proposition \ref{prop:Eggleston} to arbitrary dimensions $n\geq 3$.
  Indeed, if $C\in\CK^3$ is a $0$-symmetric polytope possessing a pair of facets $F_1, F_2$ intersecting only in a
  vertex of $C$, then each of the facets contains two different edges intersecting in the vertex. However,
  since the facets only intersect in the vertex, the four edges are different, thus
  implying that $C$ is not simple.
  Conversely, since in 3-space there is a rotational order of the facets around any vertex,
  the assumption that $C$ is non-simple implies the existence of two different facets intersecting only in a vertex.

  In dimensions at least $4$, the assumption that every pair of non-disjoint facets intersecting in at least an edge does not imply
  that $C$ is simple. Observe that the other way around still holds true: every
  two non-disjoint facets of a simple polytope intersect in an $(n-2)$-dimensional face of $P$.
  For the converse take, as an example, a double pyramid $P:=\conv((S_{n-1}\times\{0\})\cup\{\pm e^n\})$,
  where $S_{n-1}$ is an $(n-1)$-dimensional Minkowski-centered simplex and $e^n=(0,\dots,0,1)$. 
  On the one hand, $P$ is non-simple, as every vertex in $S_{n-1}\times\{0\}$ is contained in $(n-1)+2$ edges of $P$.
  On the other hand, every facet of $P$ contains exactly $n$ vertices, while $P$ has
  $n+2$ vertices in total. Thus every two different facets have
  at least $n-2$ vertices in common. Hence, if $n\geq 4$, any two
  (intersecting) facets of $P$ intersect at least in an edge.
  
\end{rem}

\begin{cor}\label{cor:S-Snonperfect}
  Let $n \ge 3$, $S,C \in \CK^n$ such that $S$ is an $n$-simplex, and $C-C=S-S$. Then $C$ is not perfect.
\end{cor}

\begin{proof}
  Because of Lemma \ref{l:gauge} \eqref{p:PerfectC-C}, it is enough to verify the corollary in case of $C=S-S$, and we may assume that, without loss of generality, $S$
  is Minkowski-centered.
  Let $S=\conv(\{p^1,\dots,p^{n+1}\})$, and $F_i=\conv(\{p^j:j\in[n+1]\setminus\{i\}\})$ be the facet of $S$ not containing $p^i$.
  Let us consider the two facets $-p^1+F_1$ and $p^2-F_2$ of $S-S$, which have the same outer normals as $F_1$ and $-F_2$, respectively,
  and $p^2-p^1$ as a common vertex.
  From Corollary \ref{cor:AsymCont} we obtain
\[
-p^1+F_1\subset(n+1)(F_1\cap(-S))\quad\text{and}\quad p^2-F_2\subset(n+1)(S\cap(-F_2)).
\]
Indeed, since $S\cap\bd(-nS)=\{p^1,\dots,p^{n+1}\}$, we have
  $F_1\cap(\bd(-nS)))=\{p^2,\dots,p^{n+1}\}$. Thus
  \[(-p^1+F_1)\cap(\relbd((n+1)(F_1\cap(-S))))=-p^1+\{p^2,\dots,p^{n+1}\}\]
  and, analogously,
  \[(p^2-F_2)\cap(\relbd((n+1)(S\cap(-F_2))))=p^2-\{p^1,p^3,\dots,p^{n+1}\}.\]
  Hence $(-p^1+F_1)\cap(p^2-F_2)=\{p^2-p^1\}$, which finishes the proof because of Lemma \ref{l:ExtEggleston}.
\end{proof}

\begin{cor}
Let $C\in\CK^n$ be 0-symmetric, $n\ge 3$, and $S$ be an $n$-simplex such that
\[S-S\subseteq D(S,C)C\subset (n+1)(S\cap(-S)).\]
Then $C$ is not perfect.
\end{cor}

\begin{proof}
Because of Proposition \ref{prop:Tcomplete} we know that in case of $D(S,C)C\neq S-S$ we have that $S$ is complete but not of constant width with respect to $C$,
showing the non-perfectness of $C$. However, in case that $D(S,C)C = S-S$, Corollary \ref{cor:S-Snonperfect} implies that $C$ is non-perfect either.
\end{proof}

Finally, the following theorem gives a characterization of perfect norms in terms of the linearity of the width between $K$ and any completion $K^*$.

\begin{thm}
Let $C\in\CK^n$. The following are equivalent:
\begin{enumerate}[(i)]
\item $C$ is perfect.
\item For all $K\in\CK^n$ and any completion $K^*$ of $K$ we have that
\[ w(\lambda K+(1-\lambda)K^*,C)=\lambda w(K,C)+(1-\lambda)w(K^*,C) \text{ for all }\lambda\in[0,1].\]
\end{enumerate}
\end{thm}

\begin{proof}
Because of Lemma \ref{l:gauge}, \eqref{p:KcomC-C} we may assume that,
without loss of generality, $C$ is 0-symmetric.
First of all, observe that
\begin{equation*}
\begin{split}
w(\lambda K+(1-\lambda)K^*,C)=\min_sb_s(\lambda K+(1-\lambda)K^*,C)=
\min_s(\lambda b_s(K,C)+(1-\lambda)b_s(K^*,C))&\\
\geq\lambda\min_sb_s(K,C)+(1-\lambda)\min_sb_s(K^*,C)=\lambda w(K,C)+(1-\lambda)w(K^*,C).&
\end{split}
\end{equation*}

If (i) holds true, the set $K^*$ is
of constant width, and hence there is equality above as $b_s(K^*,C)=w(K^*,C)$ for all $s$, therefore
proving (ii).

Now let us assume that (i) is false. Then we will prove that (ii) is false as well.
Assuming that (i) is false, there exists a complete body $U$ which is not of constant width.
The idea of the proof is to construct a subset $K$ of $U$ such that $U$ is a completion of $K$
and the width $w(K,C)$ and the width $w(U,C)$ are not achieved in the same direction.
This then implies $w(K+U,C)>w(K,C)+w(U,C)$, thus leading to the desired contradiction.

For symmetric $C$ any complete set has coinciding in- and circumcenters (see, \eg, \cite{Sall}).
Hence we may assume again that, without loss of generality,
$r(U,C)C\subset U\subset R(U,C)C$.

Now we obtain from Proposition \ref{prop:opt_hom} that there exist points
$p^i \in U\cap \bd(R(U,C)C)$ and outer normals $a^i$ of hyperplanes supporting $U$
and $R(U,C)C$ at $p^i$, $i \in [n+1]$, with $0\in\conv(\{a^1,\dots,a^{n+1}\})$. Moreover, we may scale the vectors $a^i$ such that $(a^i)^Tp^i = R(U,C)$.

Now, defining $q^i:=-(r(U,C)/R(U,C))p^i \in r(U,C)C \subset U$, we see that all the segments $[p^i,q^i]$
are diametrical chords of $U$. Hence it must hold that $b_{a^i}(U,C)=D(U,C)$ for all $i \in [n+1]$
and, defining $\beta_2 := h(r(U,C)C,a^2)=(r(U,C)/R(U,C))(a^2)^Tp^2$, also that $p^1 \in H^{\leq}_{-a^2,\beta_2}$.

Since $0=\sum\limits^{n+1}_{i=1}\lambda_ia^i$, $\lambda_i > 0 $, we may assume that $(a^2)^Tp^1< 0$.
Now consider the set $M=U\cap H^{\leq}_{\pm a^2,\beta_2}$,
which still contains $p^1$ and $q^1$.
The half-space $H^{\leq}_{-a^2,\beta_2}$ supports the inball $r(U,C)C$ and therefore contains extreme points of it.
Now, by continuity of $b_a(U,C)$, $a\in\R^n\setminus\{0\}$, for any $a$ sufficiently close to $a^2$ holds
$D(U,C)-b_a(U,C)=b_{a^2}(U,C)-b_a(U,C)<D(U,C)-w(U,C)$. However,
since the set of exposed points is dense within the set of extreme points of any convex body, we
may carefully choose $a$ such that the following are true:
\begin{itemize}
\item $w(U,C) < b_a(U,C)$,
\item $H_{\pm a^2,\beta_2}$ are hyperplanes supporting $r(U,C)C$ solely in a pair of exposed antipodal points
of the inball $r(U,C)C$ of $U$, and
\item $p^1,q^1\in H^<_{\pm a^2,\beta_2}$.
\end{itemize}
Due to these conditions, there exists $\eps>0$ small enough such that
$K:=U\cap H^{\leq}_{\pm a^2,\beta_2-\eps\norm[a]_2}$
still satisfies $p^1,q^1\in K$. Now, because $D(K,C)=D([p^1,q^1],C)=D(U,C)$, we see that
$U$ is a completion of $K$.
Moreover, since each of the half-spaces
$H^{\leq}_{\pm a^2,\beta_2-\eps\norm[a]_2}$
touches $(r(U,C)-\varepsilon)C=r(K,C)C$ in a unique exposed point and $r(K,C)C\subset\inter(U)$, $b_a(K,C)=2r(K,C)<b_s(K,C)$ for all directions $s \neq a$.
Thus the width of $K$ is uniquely attained in direction of $a$, and because
$w(U,C)<b_a(U,C)$, we can conclude
that $w(K+U,C)>w(K,C)+w(U,C)$.
\end{proof}

\bigskip
The following lemma proves the fact that all two-dimensional generalized Minkowski spaces are perfect
(\cf~\cite[p.~171]{Eg} for normed spaces).

\begin{lem}\label{l:polyimplgen}
If $K\in\CK^2$ is complete with respect to $C$, then $K-K=(D(K,C)/2)(C-C)$.
\end{lem}

\begin{proof}
Because of Lemma \ref{l:gauge}, we can assume without loss of generality that $C$ is centrally symmetric.
First we show the case in which $K$ and $C$ are polygons.
Let $x\in K$ be a point in the relative interior of the edge $E_i$ of $K$
induced by $H_{a^i,1}$. Then $b_{a^i}(K,C)=D(K,C)$
(which follows from the uniqueness of the outer normal, up to positive multiples, at that point
and the completeness of $K$) means that the edges of $K-K$ parallel
to edges of $E_i$ are contained in the boundary of $D(K,C)C$.
Since for $n=2$ all edges of $K-K$ are parallel to
edges of $P$ it follows that $K-K=D(K,C)C$.

\medskip

Now, let $K$ be an arbitrary, planar convex set,
$a$ a unit vector such that $b_a(K,C)=D(K,C)$, and
$t^j,s^j\in\R$, $j\in[2]$, such that
$H_{a,t^j}$, $H_{a,s^j}$, $j\in[2]$, are
the parallel supporting hyperplanes of $K$ and $C$ in the direction of $a$.

We then consider a sequence $C_i$ of polygons with $C\subset C_i\subset \bigcap_{j}H^\le_{a,s^j}$
and $C_i\rightarrow C$ ($i\rightarrow\infty$). Then
\[
D(K,C_i)\leq D(K,C)=b_a(K,C)=b_a(K,C_i)\leq D(K,C_i)
\]
and therefore $D(K,C_i)=D(K,C)$ for all $i$. Now we let $K_i\supset K$ be a completion of $K$
with respect to $C_i$.

We now prove by contradiction that $K_i\rightarrow K$ ($i\rightarrow\infty$).
If not, let us observe that since $K_i$ is a completion of $K$,
then $K_i\subset x+(D(K,C)/2)(C-C)$, for some $x\in K$. Hence
$\{K_i\}_{i\in\N}$ is absolutely bounded. Using the Blaschke Selection Theorem \cite[Theorem 1.8.7]{Sch},
there exists a subsequence of $K_i$
(for which, \Wlog, we may assume that it is the sequence itself) such that
$K_i\rightarrow K_0\supset K$ ($i\rightarrow\infty$).
Then on the one hand the completeness of $K$ with respect to $C$
tells us that $D(K,C)<D(K_0,C)$, and on the other hand the continuity of
the diameter implies that $\lim_{i\rightarrow\infty}D(K_i,C_i)=D(K_0,C)$.
Altogether this shows that
\[
D(K,C)<D(K_0,C)=\lim_{i\rightarrow\infty}D(K_i,C_i)=D(K,C),
\]
a contradiction. Thus $K_i\rightarrow K$ ($i\rightarrow\infty$), and hence
\[K-K=\lim_{i\rightarrow\infty}K_i-\lim_{i\rightarrow\infty}K_i=\lim_{i\rightarrow\infty}(K_i-K_i)=\lim_{i\rightarrow\infty}(D(K_i,C_i)C_i)=D(K,C)C.\]
\end{proof}

\begin{rem}
Let us observe that the argument for general $K$ and $C$ in
Lemma \ref{l:polyimplgen} uses the fact that Open Question \ref{op:compltred} is true
for polygons. Indeed, it holds true for every $0$-symmetric $C\in\CK^n$, with $n\in\N$. This therefore motivates
why examples of non-perfect norms are polytopal. Moreover, this approach might be useful as well
when considering sets that are complete and reduced simultaneously.
\end{rem}

\emph{Acknowledgements:} We would like to thank Matthias Henze for his support
in Remark \ref{rem:polytope}.

\bibliographystyle{amsplain}

\end{document}